\documentclass[11pt]{article}

\usepackage[a4paper]{geometry}
\usepackage{amsmath,amssymb,amsthm}
\usepackage{bm}
\usepackage{hyperref}
\usepackage{enumitem}
\usepackage{listings}
\usepackage{xcolor}
\usepackage{placeins}
\usepackage{graphicx}

\usepackage{tikz}\usetikzlibrary{arrows.meta,calc,decorations.pathreplacing}

\lstdefinestyle{matlabstyle}{
  language=Matlab,
  basicstyle=\ttfamily\small,
  keywordstyle=\bfseries,
  commentstyle=\itshape\color{gray!70!black},
  stringstyle=\color{teal!60!black},
  numbers=left,
  numberstyle=\tiny,
  stepnumber=1,
  numbersep=6pt,
  showstringspaces=false,
  tabsize=2,
  breaklines=true,
  frame=lines
}

\newtheorem{lemma}{Lemma}[section]

\newtheorem{proposition}{Proposition}[section]
\newtheorem{remark}{Remark}[section]
\newtheorem{example}{Example}[section]

\newcommand{\R}{{\mathbb R}}
\newcommand{\T}{{\mathbb T}}

\title{\textbf{A Moser-Type Construction for the Liouville Equation}}

\author{A. Borz\`i\thanks{Institut f\"ur Mathematik, Universit\"at W\"urzburg, Emil-Fischer-Strasse 30,
97074 W\"urzburg, Germany. e-mail: alfio.borzi@mathematik.uni-wuerzburg.de} 
\and M. Caponigro\thanks{Dipartimento di Matematica,
Università di Roma ``Tor Vergata'',
Via della Ricerca Scientifica 1,
00133 Roma,
Italy. e-mail: caponigro@mat.uniroma2.it} 
\and A. Vicari\thanks{Dipartimento di Matematica,
Università di Roma ``Tor Vergata'',
Via della Ricerca Scientifica 1,
00133 Roma,
Italy. e-mail: vicari@axp.mat.uniroma2.it} }

\date{}

\begin{document}

\maketitle

\begin{abstract}
A Moser-type construction for the kinetic Liouville equation is proposed, 
which is based on a
characteristic-adapted interpolation. For Hamiltonian accelerations,
the construction is reduced to a family of weighted elliptic problems in the
velocity variable. The corresponding kinetic compatibility condition is
derived.
\end{abstract}

\section{Introduction}

In the seminal work~\cite{moser65}, J. Moser shows that on a compact manifold without 
boundary, any pair of smooth and positive volume forms of equal total mass can be 
connected by a diffeomorphism. Later developments of this result, in particular \cite{dacorognamoser}, 
provide a constructive approach based on interpolating the densities and solving an elliptic
Neumann problem for a time-dependent vector field whose flow realizes the
desired transport.

More precisely, let $\Omega\subset \mathbb{R}^n$ be a bounded connected domain with smooth, i.e.
$C^\infty$ boundary, and let $f,g\in C^\infty(\overline{\Omega})$ satisfy
$f>0$, $g>0$ on $\overline{\Omega}$ and be such that 
$$
\int_\Omega f(x)\,dx=\int_\Omega g(x)\,dx.
$$
Then~\cite{moser65} states the 
existence of a diffeomorphism $\Gamma: \overline{\Omega} \to \overline{\Omega}$ pointwise 
fixing the boundary, such that $\Gamma_\#(f\,dx)=g\,dx$, namely
$$
f(x)=g(\Gamma(x))\det D\Gamma(x), \quad \mbox{ for every } x \in \Omega.
$$
To construct such a diffeomorphism, a possible strategy, following the notation of 
Dacorogna and Moser~\cite{dacorognamoser} (see also~\cite[Box 4.3]{santambrogio2015optimal}), is to
connect the initial and final densities by a smooth path, for instance, the linear interpolation
\begin{equation}
\rho(x,t)=(1-t) \, f(x)+t \, g(x) , \qquad t\in[0,1], \, x \in \Omega .
\label{origInterpolation}
\end{equation}
The goal is then to find a time-dependent velocity field $b(x,t)$ generating this path as a 
solution of the continuity equation
\begin{equation}\label{eq:cont}
\partial_t \rho(x,t) + \nabla \cdot \big(b(x,t) \, \rho(x,t)\big) = 0,
\qquad \rho(\cdot,0)=f, 
\end{equation}
so that $\rho(\cdot,1)=g$.  Hence, denoting by $\Gamma_t$ the flow associated with the vector field $b$,
the solution of the continuity equation \eqref{eq:cont} can be written as 
$\rho(\cdot,t)dx=(\Gamma_t)_{\#}(f\,dx)$. In particular, the $1$-time flow $\Gamma_1$ pushes 
forward the initial density onto the final one
$$
(\Gamma_1)_\# (f\,dx)=g\,dx ,
$$
providing the diffeomorphism solving the prescribed transport problem.

Differentiating~\eqref{origInterpolation}, $\partial_t\rho=g-f.$
Therefore, the problem is to construct a flux \(w:= b \, \rho \) satisfying 
$\nabla\cdot w=f-g$. This goal can be achieved setting $w=-\nabla\phi$ 
and solving the elliptic problem
\begin{equation}
\label{eq:moser_poisson}
\begin{cases}
-\Delta \phi=f - g & \text{in }\Omega,\\
\partial_n\phi=0 & \text{on }\partial\Omega,
\end{cases}
\end{equation}
(or an appropriate decay condition as $x \to \infty$) together with the normalization
\[
\int_{\Omega}\phi(x)\,dx=0.
\]

In the present work, we investigate the extension of 
this construction to the case of phase-space densities 
of 
Hamiltonian systems. In this setting, the continuity 
equation governing the evolution of densities takes the form of 
\begin{equation}\label{eqLiouville}
\partial_t \rho(x,v,t) + v\cdot\nabla_x \rho(x,v,t)  
+  \nabla_v \cdot \big( a(x,v,t) \, \rho(x,v,t)  \big)= 0, 
\end{equation}
where $(x,v)\in \Omega\subset \R^n\times\R^n$.
More precisely, the problem is formulated as follows. 
\begin{quotation}
\noindent Let $\Omega = X\times V$,
with $X=\R^n$ or $\T^n$ ($n$-dimensional torus) and $V = \R^n$. 
Let $f,g\in C^\infty(\Omega)$ be strictly positive and
such that
\[
\int_\Omega f(x,v) \,dx\,dv = \int_\Omega g(x,v) \,dx\,dv.
\]
Consider the Liouville equation~\eqref{eqLiouville} with initial condition $\rho(0)=f$. \\
Find $a:\Omega \times (0,1)    \to  \R^n$ such that the associated solution $\rho(t)$ 
satisfies $\rho(1)=g$. 
\end{quotation}
This problem can be intepreted as a controllability problem for the 
Liouville equation~\eqref{eqLiouville}.

\medskip

Equation~\eqref{eqLiouville}, called Liouville equation by L.~Boltzmann in his 
lectures~\cite{boltzmann}, plays a fundamental role in kinetic theory and in the
statistical description of non-interacting many-particle systems, 
where $\rho$ denotes the normalized material density function. 
At the microscopic level, our Liouville equation governs the evolution 
of the phase-space density of the particles' dynamics given by the 
differential system 
\begin{equation}\label{eq:micro}
\begin{cases}
\chi'(t) &=\upsilon(t), \\ 
\upsilon'(t) &=a(\chi(t),\upsilon(t),t),
\end{cases}
\end{equation}
where the spatial drift is fixed to be the velocity variable $\upsilon$,
while the force (acceleration) field constitutes the only degree of freedom.

Moser’s approach has been studied 
from various perspectives in the context of the continuity equation also
within the framework of control theory~\cite{AC09,brockett2012notes,AV25}. To the best of our knowledge, 
the corresponding problem for the Liouville equation~\eqref{eqLiouville}, defined on the 
phase space, is an open problem. It requires to resolve the significant 
constraint due to the fact that the drift in the $x$-direction is prescribed and 
equal to $v$, which is an independent variable.

\medskip

Our discussion can be placed at the intersection of several classical theories. 
The original inspiration comes from the volume-form construction of
J.~Moser~\cite{moser65} and its constructive reformulation by
B.~Dacorogna and J.~Moser~\cite{dacorognamoser}. In that framework, the
transport field is essentially arbitrary and the only condition on the initial and final densities is the
equality of total mass. The  Liouville setting differs because the
admissible phase-space velocity fields are constrained by the kinetic structure$
(\dot x,\dot v)=(v,a).$

In particular, in what follows we will focus on the conservative Hamiltonian case in which
\[
(\dot x,\dot v)=(v,-\nabla_x V).
\]
The conservative Hamiltonian case is related to the classical Liouville theorem
in Hamiltonian mechanics, already present in the works of
J.~Liouville and later systematically developed in statistical mechanics and
Hamiltonian dynamics. In this case, Hamiltonian flows preserve phase-space
volume and transport densities by measure-preserving rearrangements:
\[
g=f\circ\Psi^{-1},
\]
where \(\Psi\) is the Hamiltonian flow map. Consequently, all rearrangement
invariants are preserved. In this framework, a equimeasurability condition arises.

The distinction between arbitrary volume-preserving maps and Hamiltonian flows
is closely connected with the rigidity phenomena of symplectic geometry
initiated by Gromov through the non-squeezing theorem~\cite{Gromov1985} and
further developed in modern symplectic topology; see for instance
\cite{Hoferzehnder1994,Mcdsal2017}. In this framework, Hamiltonian diffeomorphisms
constitute a much more rigid class than general measure-preserving maps,
since they preserve not only phase-space volume but also the underlying
symplectic structure.

The preservation of quantities of the form $\int_{\Omega}\Phi(f)$ 
is also related to the theory of Casimir invariants in kinetic and fluid
equations. Such invariants play a fundamental role in the analysis of the
Vlasov equation, incompressible Euler equations, and related transport
problems; see for example \cite{MarsdenWeinstein1983,Morrison1980}. Rearrangement
methods and measure-preserving transport techniques have also been extensively
developed in optimal transport and fluid mechanics, particularly in the works
of Y.~Brenier~\cite{Brenier1989,Brenier1991,Brenier2003}.

The present discussion suggests a hierarchy of compatibility conditions for
endpoint reachability problems associated with the Liouville equation:

\begin{description}

\item[P1] for arbitrary accelerations \(a(x,v,t)\), one may expect that only
conservation of total mass is required;

\item[P2] for time-dependent conservative Hamiltonian accelerations 
$a=-\nabla_x V(x,t)$, 
equimeasurability becomes a necessary condition; (see Proposition~\ref{prop:equimeasurable} below)

\item[P3] for autonomous Hamiltonians $a=-\nabla_x V(x)$ 
the full Hamiltonian energy distribution must be preserved; (see Proposition~\ref{prop:autonomous} below)

\item[P4] for general integrable Hamiltonian systems, further invariants related
to action variables and invariant tori may appear.  
\end{description}

%
\section{The case of Hamiltonian flow}

We consider the case in which the acceleration is
generated by a time-varying potential: 
\[
        a(x,t)=-\nabla_x V(x,t).
\]
Then the phase-space dynamics is a Hamiltonian system as follows 
\begin{equation}\label{eq:hamtimedep}
\begin{cases}
\dot x &=v,\\
\dot v &=-\nabla_x V(x,t),
\end{cases}
\end{equation}
corresponding to the Hamiltonian 
$$
H(x,v,t)=\frac12 |v|^2+V(x,t).
$$
In this case a necessary condition for the reachability of $g$ from $f$ via the flow of~\eqref{eq:hamtimedep} is equimisurability as stated in the following proposition.

\begin{proposition}
\label{prop:equimeasurable}
Let $\rho(x,t)$ be the solution 
of the Liouville equation
\[
\partial_t\rho + v\cdot\nabla_x\rho - \nabla_x V(x,t)\cdot\nabla_v\rho = 0,
\qquad
\rho(\cdot,0)=f,
\]
such that 
\[
\rho(\cdot,1)=g .
\]
Then $f$ and $g$ are equimeasurable, namely
\[
        \int_{\Omega} \Phi(f(x,v))\,dx\,dv
        =
        \int_{\Omega} \Phi(g(x,v))\,dx\,dv
\]
for any measurable function $\Phi:\R \to \R$.
\end{proposition}

\begin{proof}
The phase-space vector field of~\eqref{eq:hamtimedep} is divergence-free 
since $\nabla_x\cdot v+\nabla_v\cdot(-\nabla_x V)=0$.
As a consequence, the associated flow is volume-preserving in phase space.
If \(\Psi_t\) denotes the Hamiltonian flow, then the solution of~\eqref{eqLiouville} is
\[
        \rho(t)=\rho(0)\circ \Psi_t^{-1}.
\]
The Hamiltonian flow $\Psi_t$ preserves phase-space volume
(Liouville theorem), that is,
\[
\det D\Psi_t(x,v)=1 .
\]
Along characteristics, the density satisfies
\[
\frac{d}{dt}\rho(\Psi_t(x,v),t)=0,
\]
hence $\rho(\Psi_t(x,v),t)=f(x,v)$. At time $t=1$ we obtain $g(x,v)=f(\Psi_1^{-1}(x,v))$. 

Fix $\lambda\in\R$ and define the level sets
\[
E_f(\lambda)
=
\{z\in\Omega:f(z)>\lambda\},
\qquad
E_g(\lambda)
=
\{z\in\Omega:g(z)>\lambda\}.
\]
Since $g(z)=f(\Psi_1^{-1}(z))$, we have
\[
z\in E_g(\lambda)
\iff
\Psi_1^{-1}(z)\in E_f(\lambda),
\]
that is,
\[
E_g(\lambda)=\Psi_1(E_f(\lambda)).
\]

Because $\Psi_1$ is volume-preserving,
\[
|E_g(\lambda)|
=
|\Psi_1(E_f(\lambda))|
=
|E_f(\lambda)|.
\]
Therefore
\[
|\{g>\lambda\}|
=
|\{f>\lambda\}|,
\]
and the conclusion follows.
\end{proof}

    Equimeasurability implies, in particular, that
\[
        \int_{\Omega} f^p\,dx\,dv
        =
        \int_{\Omega} g^p\,dx\,dv,
        \qquad 1\le p<\infty .
\]
Thus, in the Hamiltonian setting, equality of total mass is not
sufficient. The Hamiltonian flow preserves the full distribution of density
values.

\begin{remark}
Equimeasurability is only a necessary condition for reachability by a Hamiltonian force. 
Indeed, if
\[
g=f\circ\Psi_1^{-1},
\]
where \(\Psi_t\) is the Hamiltonian flow generated by $H(x,v,t)$, 
then \(f\) and \(g\) are equimeasurable because \(\Psi_1\) is
volume-preserving.

The converse, however, is not true in general. Equimeasurability only implies
that there exists some volume-preserving map \(T\) such that
\[
g=f\circ T^{-1},
\]
under appropriate regularity assumptions. It does not imply that \(T\) is the
time-one map of a Hamiltonian system of the mechanical form
\[
\dot x=v,\qquad \dot v=-\nabla_x V(x,t).
\]
This class is much smaller than the class of all volume-preserving maps.

Thus, the construction of \(V\) requires additional geometric compatibility, that is, 
the desired rearrangement must be realizable by a mechanical Hamiltonian flow.

This is an overdetermined constraint because the unknown force
\(-\nabla_x V(x,t)\) depends only on \(x\) and \(t\), not on \(v\).
Therefore, even when \(f\) and \(g\) are equimeasurable, a conservative
potential \(V\) may not exist.
\end{remark}

If the potential is independent of time, namely
\[
        V=V(x),
\]
then the Hamiltonian
\[
        H(x,v)=\frac12 |v|^2+V(x)
\]
is conserved along characteristics
\[
        \frac{d}{dt}H(x(t),v(t))=0.
\]
The Hamiltonian flow cannot move mass between different energy
levels. If \(g\) is reachable from \(f\) by the autonomous Hamiltonian flow,
then,  the energy marginal distributions of \(f\) and \(g\) must necessarily agree:
\[
        \int_{\{H\le E\}} f(x,v)\,dx\,dv
        =
        \int_{\{H\le E\}} g(x,v)\,dx\,dv
        \qquad \text{for all }E .
\]

\begin{proposition}\label{prop:autonomous}
Let $\rho(x,t)$ be the solution 
of the Liouville equation
\[
\partial_t\rho + v\cdot\nabla_x\rho - \nabla_x V(x)\cdot\nabla_v\rho = 0,
\qquad
\rho(\cdot,0)=f,
\]
such that 
\[
\rho(\cdot,1)=g.
\]
Let \[
        H(x,v)=\frac12 |v|^2+V(x)
\]
then for every measurable function \(\Phi : \R \to \R\), it must hold
\[
        \int_{\Omega} \Phi(H(x,v))\, f(x,v)\,dx\,dv
        =
        \int_{\Omega} \Phi(H(x,v))\, g(x,v)\,dx\,dv .
\]

\end{proposition}

Thus, in the autonomous case, equimeasurability is not sufficient. Besides
preserving phase-space volume and all \(L^p\)-norms of the density, the flow
also preserves the distribution of mass over the invariant energy surfaces
\[
        H(x,v)=E .
\]
Consequently, a time-independent conservative potential can only rearrange
density along its own Hamiltonian trajectories and cannot perform arbitrary
equimeasurable rearrangements of phase space.

\section{Characteristic-adapted interpolation}
In order to adapt the
Dacorogna-Moser constructive proof to 
the kinetic Liouville structure we consider the characteristics lines associated 
\[
        (x(t),v(t))=(x_0+t v_0,v_0) ,
\]
for chosen initial conditions $x_0$ and $v_0$,
associated 
the free transport
equation
\[
        \partial_t\rho+v\cdot\nabla_x\rho=0,
\]
corresponding to the case 
when \(a=0\) in the Liouville equation.
 Hence the free evolution of a density \(f\) is given by 
\[
        \rho(x,v,t)=f(x-tv,v).
\]
Instead of the standard Moser interpolation
we consider a
characteristic-adapted interpolation of the form
\begin{equation}
    \rho_t(x,v) = (1-t) \, f(x-tv,v) +  t \, g(x+(1-t)v,v).
    \label{newInterpolation}
\end{equation}
incorporating the
horizontal kinetic drift generated by the free characteristics. 

Notice that the endpoint conditions are satisfied: $\rho_0=f$ and $ \rho_1=g$. 
We also have
\begin{align*}
   \partial_t  \rho_t(x,v) & = - f(x-tv,v) +  g(x+(1-t)v,v) 
   - (1-t) \, v \, \partial_x f(x-tv,v) \\
   & \qquad -  t \, v \, \partial_x g(x+(1-t)v,v) \\
   & = - f(x-tv,v) +  g(x+(1-t)v,v)  - v \cdot \partial_x \rho_t (x,v) .
\end{align*}
This calculation provides the defect
\begin{equation}\label{defect}
\partial_t\rho_t+v\cdot\nabla_x\rho_t = g(x+(1-t)v,v) - f(x-tv,v) .
\end{equation}

\medskip


\medskip

Our approach is then to assume that the admissible vertical
$v$-flux \(q_t\) is generated by a potential in the velocity variable:
\[
        q_t=\rho_t\nabla_v U_t .
\]
Then the Liouville equation
\[
        \partial_t\rho_t
        +
        v\cdot\nabla_x\rho_t
        +
        \nabla_v\cdot q_t
        =
        0
\]
becomes
\begin{equation}
        \nabla_v\cdot(\rho_t\nabla_v U_t)
        = - \big( \partial_t\rho_t + v\cdot\nabla_x\rho_t  \big).
        \label{ellipticEqnew}
\end{equation}
This is a weighted elliptic equation in the velocity variable parametrized by
\((x,t)\) and its solvability will be discussed later in this section. Notice that these problems are uniformly
elliptic only if the interpolation density $\rho_t$ remains strictly positive.

\medskip


In this setting, the associated acceleration field is given by
\[
        a_t=\nabla_v U_t 
\]
and at the level of particles, the dynamics is given by the differential system
\[
        \dot x=v,
        \qquad
        \dot v=\nabla_v U_t(x,v).
\]
\begin{remark}
Assuming homogeneous Neumann boundary
conditions in \(v\) or sufficiently fast decay conditions, the solvability condition for the elliptic equation \eqref{ellipticEqnew} is given by 
\begin{equation}\label{solvcond}
        \int_V  \big(
        \partial_t\rho_t
        +
        v\cdot\nabla_x\rho_t
        \big)\,dv
        = 0 ,
\end{equation}
for every fixed \(x\) and \(t\).

Thus the interpolation path \(\rho_t\) must satisfy the
following kinetic compatibility condition
\begin{equation*}
        \partial_t\int_V\rho_t\,dv +  \nabla_x\cdot\int_V v\rho_t\,dv  = 0 .
\end{equation*}
If we introduce the spatial density
\[
        m_t(x)=\int_V\rho_t(x,v)\,dv
\]
and the momentum density
\[
        j_t(x)=\int_V v\rho_t(x,v)\,dv ,
\]
the above condition becomes the macroscopic continuity equation
\[
        \partial_t m_t+\nabla_x\cdot j_t=0 .
\]
\end{remark}

\begin{remark} 
Because of \eqref{defect} and the solvability condition for the elliptic problem \eqref{solvcond}, we have that $f$ and $g$ must satisfies 
\[
        \int_V f(x-tv,v)\,dv
        =
        \int_V g(x+(1-t)v,v)\,dv .
\]
Unlike the classical Dacorogna-Moser construction, where solvability reduces
to equality of total mass, the present kinetic setting imposes a family of
characteristic marginal constraints indexed by \(t\in[0,1]\).\\
However we stress that this condition originates from the assumption 
that the admissible vertical $v$-flux \(q_t\) is generated 
by a potential in the velocity variable and from the choice of the characteristc-adapted interpolation \eqref{newInterpolation}. Therefore it is not a 
necessary condition for the existence of a generic suitable acceleration sending $f$ to $g$.  
\end{remark}
\medskip




We now investigate solvability of the weighted elliptic problem in \eqref{ellipticEqnew}. 
\begin{proposition}
Let $f,g\in C^\infty(\Omega)$ satisfy
$f>0$, $g>0$ on $\Omega$ and be such that
\[
\int_{\R^n} f(x-tv,v)\,dv = \int_{\R^n} g(x+(1-t)v,v)\,dv,
\]
for every $x\in X,t\in [0,1]$. 
Fix $x\in X,t\in [0,1]$ and consider  the measure 
\begin{equation}\label{measure}
d\mu_t^x(v)
=
\rho_t(x,v)\,dv ,
\end{equation}
where
\[
\rho_t(x,v)
=
(1-t)f(x-tv,v)
+
t\,g(x+(1-t)v,v),
\]
and the associated weighted Sobolev space
\begin{equation}\label{wsob}
H^1(\mu_t^x)
=
\left\{
u :\R^n \to \R \,\,\,\,\, \mathrm{s.t.}\,
\int_{\R^n} |u|^2\,d\mu_t^x
+
\int_{\R^n} |\nabla_v u|^2\,d\mu_t^x
<
\infty
\right\}.
\end{equation}
If, for every $x\in X, t\in[0,1]$, there exists a constant $C_P^{x,t}$ such that, defining 
\[
\bar u
=
\frac{\int_{\R^n} u\,d\mu_t^x}
{\int_{\R^n} d\mu_t^x},
\]
the Poincar\'e inequality 
\begin{equation}
\int_{\R^n}
|u-\bar u|^2\,d\mu_t^x
\le
C_P^{x,t}
\int_{\R^n}
|\nabla_v u|^2\,d\mu_t^x ,
\label{ePFinequality}
\end{equation}
holds for every $u \in H^1(\mu_t^x)$, then
there exists $a:\Omega \times [0,1] \to  \R^n$ such that the associated solution 
of the Liouville equation~\eqref{eqLiouville} is exactly $\rho_t$. In particular, it satisfies $\rho_0=f$
and $\rho_1=g$. 
\end{proposition}

The proof splits into two steps. In Lemma \ref{step1}, we show the relation between the solution of the continuity equation, in which the acceleration $a(x,v,t)$ is chosen to be generated by a potential $U_t^x(v)$, and the elliptic equation that this potential must satisfy. In Lemma \ref{step2}, we show that under the assumptions of the proposition, the elliptic problem admits a unique solution, modulo additive constants. Then the conclusion follows taking $a(x,v,t)=\nabla_vU_t^x(v)$. 
\begin{lemma}\label{step1}
Let $U_t^x$ be a solution of  
\[
\nabla_v\cdot(\rho_t\nabla_v U_t^x) = f(x-tv,v)-g(x+(1-t)v,v).
\]
Then $\rho_t$ solves the Liouville equation~\eqref{eqLiouville} with
\[
a(x,v,t)=\nabla_v U_t^x(v).
\]
In particular, it satisfies $\rho_0=f$
and $\rho_1=g$. 
\end{lemma}

\begin{proof}
Recall that
$$
    \rho_t(x,v) = (1-t) \, f(x-tv,v) +  t \, g(x+(1-t)v,v).
$$
Hence
\begin{align*}
   \partial_t  \rho_t(x,v) 
& = - f(x-tv,v) +  g(x+(1-t)v,v)     - (1-t) \, v \, \partial_x f(x-tv,v) \\
   & \qquad -  t \, v \, \partial_x g(x+(1-t)v,v) \\
   & = - f(x-tv,v) +  g(x+(1-t)v,v)  - v \cdot \nabla_x \rho_t (x,v).
\end{align*}
Which gives
\begin{align*}
\partial_t  \rho_t(x,v) + v \cdot \nabla_x \rho_t (x,v) & = g(x+(1-t)v,v) - f(x-tv,v)\\
& =  
- \nabla_v\cdot(\rho_t\nabla_v U_t^x).
\end{align*}
Then
$$
 \partial_t\rho_t+v\cdot\nabla_x\rho_t  + \nabla_v a_t(x,v) = 0.
$$
\end{proof}

\begin{lemma}\label{step2}
Consider  the measure $d\mu_t^x(v)$ and the associated weighted Sobolev space $H^1(\mu_t^x)$ as defined in \eqref{measure},\eqref{wsob}.

Consider the equation
\begin{equation}\label{eq:elliptic}
\nabla_v\cdot(\rho_t(x,v)\nabla_v U_t^x(x,v))
=
h_t(x,v),
\end{equation}
where
\[
h_t(x,v) = f(x-tv,v)-g(x+(1-t)v,v).
\]
Assume that, for every $x \in X$ and $t\in[0,1]$, the Poincar\'e inequality \eqref{ePFinequality} holds for every $u \in H^1(\mu_t^x)$.
Then~\eqref{eq:elliptic} with homogeneous Neumann boundary conditions in $v$ or sufficiently fast decay conditions has a unique solution in the space
\[
\left\{
u\in H^1(\mu_t^x):
\int_{\R^n}u\,d\mu_t^x=0
\right\}.
\]
\end{lemma}
\begin{proof}

The weak formulation of the elliptic problem~\eqref{eq:elliptic} is
\[
\int_{\R^n}
\nabla_v U_t^x\cdot\nabla_v\varphi
\,d\mu_t^x
=
- \int_{\R^n}
h_t\varphi\,dv,
\qquad \varphi\in H^1(\mu_t^x).
\]
Since constants belong formally to the kernel of the operator, uniqueness
can only hold modulo additive constants. Therefore for uniqueness 
we impose the normalization condition
\begin{equation}
\int_{\R^n} U_t^x\,d\mu_t^x=0 .
\label{eUniqCond}
\end{equation}
The Poincar\'e
inequality ensures that the bilinear form
\[
B(u,\varphi)
=
\int_{\R^n}
\nabla_v u\cdot\nabla_v\varphi
\,d\mu_t^x
\]
is coercive on the subspace
\[
\left\{
u\in H^1(\mu_t^x):
\int_{\R^n}u\,d\mu_t^x=0
\right\}.
\]
Consequently, by the Lax--Milgram theorem, one obtains existence and
uniqueness of weak solutions. 
\end{proof}

\begin{example}
The essential analytical issue becomes the validity of the
weighted Poincar\'e inequality associated with the interpolation density
\(\rho_t\).
A classical sufficient condition is provided by Gaussian confinement.
Assume that there exist positive constants \(c_0,c_1,C_0,c_2\) such that for any $x \in X$ and $t \in [0,1]$
\[
c_0 e^{-c_1|v|^2}
\le
\rho_t(x,v)
\le
C_0 e^{-c_2|v|^2},
\qquad
v\in\R^n.
\]

Then the measure \(d\mu_t^x=\rho_t\,dv\) is uniformly comparable with a
Gaussian measure. Consequently, classical weighted Poincar\'e inequalities
for Gaussian-type measures apply; see for instance the Bakry-\'Emery theory
\cite{BakryEmery1985}. Therefore there exists a constant \(C>0\),
depending only on the comparison constants
\(c_0,c_1,C_0,c_2\), such that \eqref{ePFinequality} holds. 

We now show that such Gaussian confinement estimates hold for suitable
Gaussian endpoint densities. Consider the characteristic interpolation
\[
\rho_t(x,v)
=
(1-t)f(x-tv,v)
+
t\,g(x+(1-t)v,v).
\]
Assume that the endpoint densities satisfy Gaussian bounds of the form
\[
a_f e^{-\alpha_f|v|^2-\beta_f}
\le
f(x-tv,v)
\le
A_f e^{-\gamma_f|v|^2+B_f},
\]
and
\[
a_g e^{-\alpha_g|v|^2-\beta_g}
\le
g(x+(1-t)v,v)
\le
A_g e^{-\gamma_g|v|^2+B_g},
\]
for any \(x\), \(v\), and \(t\in[0,1]\).

We first derive the lower bound. If \(0\le t\le\tfrac12\), then
\[
\rho_t(x,v)
\ge
\frac12 f(x-tv,v)
\ge
\frac12 a_f e^{-\alpha_f|v|^2-\beta_f}.
\]

If \(\tfrac12\le t\le1\), then
\[
\rho_t(x,v)
\ge
\frac12 g(x+(1-t)v,v)
\ge
\frac12 a_g e^{-\alpha_g|v|^2-\beta_g}.
\]

Consequently, we obtain
\[
\rho_t(x,v) \ge c_0 e^{-c_1|v|^2},
\]
with
\[
c_0 = \frac12 \min\{a_f e^{-\beta_f},a_g e^{-\beta_g}\}, \qquad 
c_1 = \max\{\alpha_f,\alpha_g\}.
\]

We now derive the upper bound. Since $t \in [0,1]$, we have
\begin{align*}
\rho_t(x,v)
&=
(1-t)f(x-tv,v)
+
t\,g(x+(1-t)v,v)
\\
&\le
f(x-tv,v)
+
g(x+(1-t)v,v)
\\
&\le
A_f e^{-\gamma_f|v|^2+B_f}
+
A_g e^{-\gamma_g|v|^2+B_g}.
\end{align*}

Hence we obtain
\[
\rho_t(x,v)
\le
C_0 e^{-c_2|v|^2},
\]
where
\[
c_2
=
\min\{\gamma_f,\gamma_g\}, \qquad
C_0 = A_f e^{B_f}+A_g e^{B_g}.
\]

Therefore the interpolation density satisfies the two-sided Gaussian estimate
\[
c_0 e^{-c_1|v|^2} \le \rho_t(x,v) \le C_0 e^{-c_2|v|^2}.
\]

In particular, the weighted elliptic problem associated with the
characteristic interpolation admits a coercive variational formulation in
weighted Sobolev spaces, yielding together with \eqref{eUniqCond} existence 
and uniqueness of weak solutions.
\end{example}

\section{Further remarks}

The preceding construction shows that a Moser-type strategy for the
Liouville equation is possible once an interpolation path
\(\rho_t\) has been prescribed. However, unlike the classical
Dacorogna-Moser construction for the continuity equation, the interpolation
cannot be regarded as an auxiliary independent choice, but rather as part of the problem.

In the Liouville equation,
\[
\partial_t\rho
+
v\cdot\nabla_x\rho
+
\nabla_v\cdot(\rho a)=0,
\]
the phase-space characteristics are constrained by
\[
\dot x=v,
\qquad
\dot v=a(x,v,t).
\]
Thus the horizontal transport is prescribed by the velocity variable itself,
and only the vertical component is free. Consequently, fixing an
interpolation path already imposes a strong restriction on the admissible
characteristic geometry.

After prescribing \(\rho_t\), the Liouville equation becomes
\[
\nabla_v\cdot(\rho_t a_t)
=
-
\big(
\partial_t\rho_t
+
v\cdot\nabla_x\rho_t
\big).
\]
Under the potential ansatz $a_t=\nabla_v U_t$, 
one obtains the weighted elliptic equation
\[
\nabla_v\cdot(\rho_t\nabla_v U_t)
=
-
\big(
\partial_t\rho_t
+
v\cdot\nabla_x\rho_t
\big).
\]
Assuming homogeneous Neumann boundary conditions in $v$ or sufficiently fast decay, its  solvability condition is
\[
\int_V
\big(
\partial_t\rho_t
+
v\cdot\nabla_x\rho_t
\big)\,dv
=
0 ,
\]
or as we have seen
\[
\partial_t m_t+\nabla_x\cdot j_t=0 .
\]

Therefore the compatibility condition arising from the potential ansatz is not merely a property of the
endpoints \(f\) and \(g\), but of the entire interpolation path. 

In particular, choosing the characteristics interpolation
\[
\rho_t(x,v)
=
(1-t)f(x-tv,v)
+
t g(x+(1-t)v,v), 
\]
we have seen that the solvability condition reduces to the property of the endpoints of satisfying the equality
\[
        \int_V f(x-tv,v)\,dv
        =
        \int_V g(x+(1-t)v,v)\,dv,
\]
for any $x\in X$, $t \in [0,1]$. 

In any case, failure of the associated compatibility condition does not
imply failure of Liouville reachability. It only implies failure of this
particular Moser-type construction based on the potential ansatz and on this characteristic-adapted interpolation.

This suggests that a genuinely general Moser construction for the Liouville
equation should treat the interpolation itself as part of the unknown.\\

\bibliographystyle{abbrv}
\bibliography{bibliomoser}

\end{document}